\documentclass[twoside,12pt,leqno]{article}

\usepackage{amsthm, amssymb, latexsym, amsmath, amscd, array, graphicx, enumerate, lmodern, slantsc, hyperref,amsfonts,mathtools}
\usepackage[english]{babel}
\usepackage[all]{xy}
\usepackage{mathtools}

\CompileMatrices

\usepackage{color}

\def\RP{\protect\operatorname{\mathbb{R}P}}
\def\TC{\protect\operatorname{TC}}

\def\cat{\protect\operatorname{cat}}
\def\cbe{\protect\operatorname{cbe}}

\def\zcl{\protect\operatorname{zcl}}
\def\hdim{\protect\operatorname{hdim}}

\def\max{\protect\operatorname{max}}

\def\Imm{\protect\operatorname{Imm}}

\newtheorem{proposition}{Proposition}[section]
\newtheorem{definition}[proposition]{Definition}
\newtheorem{theo}[proposition]{Theorem}
\newtheorem{remark}[proposition]{Remark}

\newtheorem{example}[proposition]{Example}

\newtheorem{conjecture}[proposition]{Conjecture}

\makeindex

\title{The stability of the higher topological complexity of real projective spaces: an approach to their immersion dimension}
\author{Natalia Cadavid-Aguilar, Jes\'us Gonz\'alez\thanks{Partially supported by Conacyt Research Grant 221221.} \hspace{.2mm}, and Aldo Guzm\'an-S\'aenz}

\date{\today}

\begin{document}

\maketitle

\begin{abstract}
The $s$-th higher topological complexity of a space $X$, $\TC_s(X)$, can be estimated from above by homotopical methods, and from below by homological methods. We give a thorough analysis of the gap between such estimates when $X=\RP^m$, the real projective space of dimension $m$. In particular, we describe a number $r(m)$, which depends on the structure of zeros and ones in the binary expansion of $m$, and with the property that $\TC_s(\RP^m)$ is given by $sm$ with an error of at most one provided $s\geq r(m)$ and $m\not\equiv3\bmod4$ (the error vanishes for even $m$). The latter fact appears to be closely related to the estimation of the Euclidean immersion dimension of $\RP^m$. We illustrate the phenomenon in the case $m=3\cdot2^a$.
\end{abstract}

\medskip
\noindent{{\it 2010 Mathematics Subject Classification}:  55M30, 57R42, 68T40, 70B15.}

\noindent{{\it Keywords and phrases:} Higher topological complexity, zero-divisors cup-length, Euclidean immersion problem, real projective space.}

\tableofcontents

\section{Introduction and motivation}\label{secintro}
Farber's notion of topological complexity ($\TC$) was introduced in~\cite{Far,MR2074919} as a way to study the motion planning problem in robotics from a topological perspective. Due to its homotopy invariance, the concept quickly captured the attention of algebraic topologists who began to study the homotopy $\TC$-phenomenology. In particular, Farber's $\TC$ was soon identified as a special instance of a slightly more general concept. Rudyak's sequential (also known as ``higher'') topological complexity $\TC_s$, which recovers Farber's $\TC$ if $s=2$, can be thought of as a measure of the robustness to noise of motion planning algorithms in automated multitasking processes~(\cite{bgrt,Ru10}). Rudyak's $\TC_s$ resembles Farber's $\TC_2$ in many respects, and a number of properties and computations for $\TC_2$ can be carried over in the $\TC_s$ realm. Yet, some subtle differences between the original and sequential concepts arise~(see for instance~\cite{GGGL}). The goal of this paper is to study one such difference, which seems to lie right at the heart of the problem of finding optimal immersions of manifolds into Euclidean spaces.

\medskip
A key connection of the $\TC$-ideas with the mathematics from last century's golden age in topology came at an early stage in the $\TC$-development. As shown in~\cite{MR1988783}, the calculation of the topological complexity of a real projective space $\RP^m$ is equivalent to the determination of the smallest dimension of Euclidean spaces where $\RP^m$ can be immersed: $\TC(\RP^m)=\Imm(\RP^m)$\footnote{This equality is off by one in the three cases where $\RP^m$ is parallelizable.}. In retrospective, the immersion problem of real projective spaces in particular, and of manifolds in general, triggered much of the algebraic topology developments in the 1970's and 1980's. However, it is rather ironic that, despite the fact that algebraic topology is currently a highly sophisticated and active research field, the Euclidean immersion problem of real projective spaces stands as an open (and particularly difficult) challenge that has seen only scattered progress in the last 20 years. In very crude terms, the evidence collected through more than 65 years of research on the subject suggests a relation of the form $\TC_2(\RP^m)=2m-\delta(m)$ for some function $\delta(m)$ whose first order of approximation would be given by $2\alpha(m)$---twice the number of ones appearing in the binary expansion of $m$. Indeed, in the mid 1980's, after a period packed with new immersion and non-immersion results for $\RP^m$, some experts believed\footnote{See the Mathematical Reviews' review for~\cite{MR769162}.} that the optimal Euclidean immersion of $\RP^m$ (that is, the value of $\TC_2(\RP^m)$) would be in dimension
\begin{equation}\label{belief80s}
\Imm(\RP^m)=2m-2\alpha(m)+o(\alpha(m)).
\end{equation}

Such a state of affairs is in high contrast with the situation for the (moderately) higher topological complexity of real projective spaces. It is standard that $\TC_s(\RP^m)=sm-\delta_s(m)$ for some non-negative integer $\delta_s(m)$. While $\delta_2=\delta$, the key actor in the previous paragraph, it is shown in~\cite{GGL} by cohomological methods that
\begin{equation}\label{bienacotados}
\mbox{$\delta_s(m)\in\{0,1\}$ for $m\not\equiv3\bmod4$,}
\end{equation}
provided $s\geq\ell(m)$, where
$$
\ell(m)=\begin{cases} m+1,& \mbox{if $m$ is even;}\\ 
\frac{m+1}2, & \mbox{if $m\equiv1\bmod4$,}
\end{cases}
$$
and in fact 
\begin{equation}\label{andinfact}
\mbox{$\delta_s(m)=0$, \ if $m$ is even}
\end{equation}
(and $s\geq\ell(m))$. In the light of such a result, it is natural to propose that a reasonable understanding of the nature of $\delta_2$ (at least for $m\not\equiv3\bmod4$) would follow from a global understanding of the increasing behavior of the finite sequence of numbers
\begin{equation}\label{sqnm}
\delta_{\ell(m)}(m),\;\delta_{\ell(m)-1}(m),\;\ldots,\;\delta_2(m).
\end{equation}
The weakness of such a proposal steams from the fact that the length of the sequence~(\ref{sqnm}) is linear in $m$, while we have pointed out that $\delta_2(m)$ is expected to have (big $O$, at least) order $\alpha(m)$. This means that, even if the sequence~(\ref{sqnm}) was monotonic, many pairs of its consecutive elements would have to be equal. Therefore, the task of understanding $\delta_2(m)$ would have only been replaced by the apparently harder problem of deciding which instances of consecutive pairs of elements in~(\ref{sqnm}) differ ---and by how much they differ. 

\medskip
The goal of this paper is to sustain the proposal in the previous paragraph by showing that, in fact, a large portion of the initial elements in the sequence~(\ref{sqnm}) remain well-controlled in the sense that they satisfy~(\ref{bienacotados}) and~(\ref{andinfact}). More specifically, consider the tail elements in~(\ref{sqnm})
\begin{equation}\label{sqnmchica}
\delta_{\lambda(m)}(m),\; \delta_{\lambda(m)-1}(m),\;\ldots,\;\delta_2(m)
\end{equation}
that are \emph{not} well-controlled (in the above sense). Unlike the linear function $\ell(m)$,  $\lambda(m)$ is a much smaller function; it is estimated from above (in Theorem~\ref{nataliathm}) by a function $r(m)$ which depends in a subtle manner on the number \emph{and distribution} of ones in the binary expansion of $m$. Remark~\ref{estimacion} below discusses the interconnection between the functions $\lambda(m)$ and $r(m)$, while the series of propositions following Definition~\ref{ksdiwjdnc} in Section~\ref{sectionbehavior} sample the nature of the latter function. In addition, we provide evidence suggesting that the sequence~(\ref{sqnmchica}) actually has a nicely regular monotonic increasing behavior in some cases. The moral is, then, that a global understanding of $\Imm(\RP^m)$ would come from a reasonable understanding of the increasing rate of the terms in the sequence~(\ref{sqnmchica}). See Conjecture~\ref{noimmpot} in Section~\ref{sectionapproach} and its subsequent discussion for a sample of how a potentially nice behavior of the sequence~(\ref{sqnmchica}) would lead to new and interesting (non-)immersion results for real projective spaces. This paper thus represents the initial step in the authors' program to study the mysterious function $\delta(m)$ through the sequence of (increasingly more complicated) numbers $$\TC_{\lambda(m)}(\RP^m)\geq\TC_{\lambda(m)-1}(\RP^m)\geq\cdots\geq\TC_2(\RP^m)=\Imm(\RP^m).$$

Section~\ref{sectionpreliminaries} of this paper reviews the basic $\TC_s$ definitions and auxiliary results we need. Section~\ref{sectioncohomology} deals with the cohomology component in the estimation of $\lambda(m)$. 

\section{Preliminaries}
\label{sectionpreliminaries}
For an integer $s\geq2$, the \emph{$s$-th higher} (also referred to as \emph{sequential}) \emph{topological complexity} of a path connected space $X$, $\TC_{s}(X)$, is defined by Rudyak in~\cite{Ru10} as the reduced Schwarz genus of the fibration
\begin{equation}\label{evaluationmaps}
e_s=e^X_{s}:X^{J_{s}}\to X^{s}.
\end{equation}
Here $J_s$ is the wedge sum of $s$ (ordered) copies of the interval $[0,1]$, where $0\in[0,1]$ is the base point for the wedge, and $$e_s(f_1, \ldots, f_s)=(f_1(1), \ldots, f_s(1)), \quad \quad \mbox{for}  \quad  (f_1, \ldots, f_s) \in X^{J_s},$$ is the map evaluating at the extremes of each interval. Thus $\TC_{s}(X)+1$ is the smallest cardinality of open covers $\{U_i\}_i$ of $X^s$ so that $e_s$ admits a continuous section $\sigma_i$ on each $U_i$. The open sets $U_i$ in such an open cover are called {\it local domains}, the corresponding sections $\sigma_i$ are called {\it local rules}, and the resulting family of pairs $\{(U_i,\sigma_i)\}$ is called a {\it motion planner}. We say that such a family is an {\it optimal motion planner} if it has $\TC_{s}(X)+1$ local domains. This number is a generalization of the concept of topological complexity introduced by Farber in \cite{Far} as a model to study the continuity instabilities in the motion planning of an autonomous system (robot) whose space of configurations is $X$. The term ``higher'' comes by considering the base space $X^s$ of $e_s$ as a series of prescribed stages in the robot motion planning, while Farber's original case $s=2$ deals only with the space $X\times X$ of initial-final stages.

\medskip
Most of the existing methods to estimate the topological complexity of a given space are cohomological in nature and are based on some form of obstruction theory. One of the most successful such methods, as described in~\cite[Theorem~3.9]{bgrt}, is: 

\begin{proposition}\label{ulbTCn}
For a $c$-connected space $X$, $\zcl_s(X)\leq\TC_s(X)\leq s\hdim(X)/(c+1).$
\end{proposition}

The notation $\hdim(X)$ stands for the (cellular) homotopy dimension of $X$, i.e.~the minimal dimension of CW complexes having the homotopy type of $X$. The $s$-th zero-divisor cup-length of $X$, $\zcl_s(X)$, is defined in purely cohomological terms. Given a commutative ring $R$, $\zcl_s(X)$ is the maximal number of elements in ker($\Delta_s^*\colon H^*(X^s;R)\to H^*(X;R)$) having a non-trivial product, where $\Delta_s\colon X \to X^s$ is the $s$-fold iterated diagonal. In this work, we will only be concerned with simple coefficients in $R=\mathbb{Z}_2$, and will omit reference of coefficients in writing a cohomology group $H^*(X)$. In these terms, $\Delta_s^*\colon H^*(X^s)=H^*(X)^{\otimes s}\to H^*(X)$ is given by the $s$-fold iterated cup-multiplication, which explains the notation ``zcl" (zero-divisors cup-length) for elements in the kernel of~$\Delta^*_s$.

\section{Cohomology input}
\label{sectioncohomology}
Recall from~\cite[Corollary~3.3]{bgrt} the inequality $\cat(X^{s-1})\leq\TC_s(X)\leq\cat(X^s)$. Since $\cat((\RP^m)^s)=sm$ for any $s$, the monotonic sequence $\TC_2(\RP^m)\leq\TC_3(\RP^m)\leq\cdots\leq\TC_s(\RP^m)\leq\cdots$ has an average linear growth. This section's goal is to study the actual deviation of the above growth from the linear function $sm$.

\begin{example}\label{hopfspaces}{\em
$\RP^1$, $\RP^3$, and $\RP^7$ are Hopf spaces, so that~\cite[Theorem~1]{MR3020869} gives $\TC_s(\RP^m)=\cat((\RP^m)^{s-1})=m(s-1)$ for any $s$ if $m\in\{1,3,7\}$.
}\end{example}

\medskip
For $m\geq1$ and $s\geq2$, set $$G_s(m)=sm-\zcl_s(\RP^m),$$ the gap in the estimate in Propositon~\ref{ulbTCn} for $\TC_s(\RP^m)$, and let $e(m)$ stand for the length of the block of consecutive ones ending the binary expansion of $m$, so that $e(m)=0$ is $m$ is even. In other words, $e(m)$ is defined by the formula
\begin{equation}\label{colaE}
m\equiv2^{e(m)}-1\bmod2^{e(m)+1}.
\end{equation}
Then, as shown in~\cite[Lemma~4.3 and Theorem~4.4]{GGL}, the sequence of non-negative integers $\{G_s(m)\}_s$ is non-increasing,
$$
G_1(m)\geq G_2(m)\geq G_3(m)\geq\cdots\geq0,
$$
and stabilizes to some non-negative integer $G(m)$ which is bounded from above by $2^{e(m)}-1$. For instance, $G(m)=0$ for all even $m$, which implies, in view of Proposition~\ref{ulbTCn}, that
\begin{equation}\label{enellimite}
\TC_s(\RP^m)=\zcl_s(\RP^m)=sm
\end{equation}
for $s$ large enough. Likewise, if $m\equiv1\bmod4$ (so $e(m)=1$), then $0\leq G(m)\leq1$, and we get
$$
sm-1\leq\TC_s(\RP^m)\leq sm
$$
again for large enough $s$. More generally, as indicated in the introduction, we let $\delta_s(m)$ be defined by the formula
$$
\TC_s(\RP^m)=sm-\delta_s(m),
$$
so that
\begin{equation}\label{acotaditos}
0\leq\delta_s(m)\leq G_s(m),
\end{equation}
in view of Proposition~\ref{ulbTCn}. 

\medskip
Although taking small values, the functions $\delta_s$ are notably difficult to deal with (specially for small values of $s$) as they reflect the intrinsic homotopy phenomenology of the multi-sectioning problem for the fibrations~(\ref{evaluationmaps}). A more accessible task is to deal with the functions $G_s$ since, by construction, these objects depend only on the mod 2 cohomology ring of $\RP^m$. However, in a large portion of the cases (e.g.~those noted in~(\ref{enellimite})), we have $\delta_s(m)=G_s(m)$, which justifies a careful analysis of the functions~$G_s$. Naturally, two central tasks in such a direction are $(i)$ the computation of the stabilized $G(m)$, and $(ii)$ the determination of the smallest integer $s(m)\geq2$ satisfying
\begin{equation}\label{concretars0m}
G_s(m)=G(m)\;\;\,\mbox{for \ $s\geq s(m)$.}
\end{equation}

Both tasks are addressed in this work: We have mentioned that $G(m)\leq2^{e(m)}-1$, and evidence has been given in~\cite{GGL} (see Example~\ref{el de la evidencia} below) to conjecture that, in fact, $G(m)=2^{e(m)}-1$ for all $m$. The conjecture is proved in Theorem~\ref{estabilizacion} below. On the other hand, $s(m)$ is estimated in Theorem~\ref{nataliathm} by a function whose value in $m$ depends strongly on the number and distribution of ones in the dyadic expansion of $m$ (see Section~\ref{sectionbehavior} for concrete examples). 

\begin{example}\label{el de la evidencia}
{\em It follows from~\cite[Example~4.6]{GGL} that $G(2^e-1)=2^e-1$ and $s(2^e-1)=2$ for any $e\geq1$. Such a fact can be thought of as suggesting that the hardness of computing $\TC_s(\RP^{2^e-1})$ is independent of~$s$. (This should be compared with the general belief that the hardest instances in the Euclidean immersion problem for real projective spaces $\RP^{m}$ arise when $m+1$ is a 2-power.) In the same direction, the moral in Theorem~\ref{nataliathm} below is that, as long as $s$ is moderately large enough, the computation of $\TC_s(\RP^m)$ would be as hard as the computation of $\TC_s(\RP^{2^{e(m)}-1})$. Of course, the smaller $e(m)$ is, the more useful the lesson becomes.
}\end{example}

\begin{theo}\label{estabilizacion}
With the notation in~(\ref{colaE}), $G(m)=2^{e(m)}-1$.
\end{theo}
The proof is given after the statement and proof of Theorem~\ref{nataliathm} below.

\begin{definition}\label{nataliadef}
Let $m$ be a positive integer such that $m+1$ is not a 2-power, and set $e=e(m)$. Let $k$ be the first positive integer with $2^k>m$ (so $k>e$), and set $d_0=2^k-m-1$ (so $d_0$ is a positive integer divisible by $2^e$). Consider the non-negative integer $t=(d_0-2^e)/2^e$ and, for $1\leq\ell\leq t$, set $d_\ell=d_0-2^e \ell$ (so $d_0>d_1>\cdots>d_t=2^e$). Define non-negative integers $r_\ell$ $(0\leq\ell\leq t)$ by the recursive equations
\begin{align*}
r_0=&\begin{cases}
\left\lfloor\frac{m-(2^e-1)}{d_0}\right\rfloor, &
\mbox{if $\;\binom{m+d_0}{d_0}\equiv1\bmod2$;} \\
0\,, & \mbox{otherwise,}
\end{cases}\\
r_1=&\begin{cases}
\left\lfloor\frac{m-(2^e-1)-d_0r_0}{d_1}\right\rfloor, &
\mbox{if $\;\binom{m+d_1}{d_1}\equiv1\bmod2$;} \\
0\,, & \mbox{otherwise,}
\end{cases}\\
r_2=&\begin{cases}
\left\lfloor\frac{m-(2^e-1)-d_0r_0-d_1r_1}{d_2}\right\rfloor, &
\mbox{if $\;\binom{m+d_2}{d_2}\equiv1\bmod2$;} \\
0\,, & \mbox{otherwise,}
\end{cases}\\
\cdots&\\
r_t=&\begin{cases}
\left\lfloor\frac{m-(2^e-1)-d_0r_0-d_1r_1-\cdots-d_{t-1}r_{t-1}}{d_t}\right\rfloor, &
\mbox{if $\;\binom{m+d_t}{d_t}\equiv1\bmod2$;} \\
0\,, & \mbox{otherwise.}
\end{cases}
\end{align*}
Lastly, set $r(m)=1+\sum_{\ell=0}^{t}r_\ell$.
\end{definition}

In Definition~\ref{nataliadef}, the dyadic expansion of $d_0$ is the ``complement'' of that for $m$. So $\binom{m+d_0}{d_0}$ is odd (and thus $r_0=\lfloor\frac{m-(2^e-1)}{d_0}\rfloor$). Further, $d_t=2^e$ and $m\equiv2^e-1\pmod{2^{e+1}}$, so that the binomial coefficient $\binom{m+d_t}{d_t}$ is odd too. In addition, since $2^e$ divides $m-(2^e-1)$ as well as each $d_\ell$, we actually have
\begin{equation}\label{rt}
r_t=\frac{m-(2^e-1)-d_0r_0-d_1r_1-\cdots-d_{t-1}r_{t-1}}{d_t}.
\end{equation}

\begin{theo}\label{nataliathm}
With the notation in~(\ref{concretars0m}) and Definition~\ref{nataliadef}, $s(m)\leq r(m)$.
\end{theo}

Definition~\ref{nataliadef} and Theorems~\ref{estabilizacion} and~\ref{nataliathm} were suggested by a large amount of computer calculations. In fact, on the basis of the computational evidence, we conjecture that the conclusion in Theorem~\ref{nataliathm} can be strengthened to the equality $s(m)=r(m)$.

\begin{remark}\label{estimacion}{\em
Recall the function $\lambda(m)$ introduced in~(\ref{sqnmchica}) to denote the smallest integer $s$ such that $\delta_s(m)$ is not ``well-controlled''. Theorem~\ref{estabilizacion} and~(\ref{acotaditos}) allow us to pin down the idea. Namely, $\lambda(m)$ stands for the smallest integer such that $\delta_s(m)\leq2^{e(m)}-1$. Of course, the control is tighter whenever $m\not\equiv3$ mod~4. For instance, if $m$ is even, then $\delta_{\lambda(m)}(m)=0$ while $\delta_s(m)>0$ for $2\leq s<\lambda(m)$. In any case,~(\ref{acotaditos}) and Theorems~\ref{estabilizacion} and~\ref{nataliathm} give
\begin{equation}\label{3elties}
\lambda(m)\leq s(m)\leq r(m).
\end{equation}
Further, as suggested in the final section in this paper, all three inequalities in~(\ref{3elties}) seem to be sharp when $m=3\cdot2^a$ ---and, possibly, in many more instances. 
}\end{remark}

The proofs of Theorems~\ref{estabilizacion} and~\ref{nataliathm} use the following notation. For $s\geq2$, let $x_i\in H^1((\RP^m)^s)$ be the pull back of the non-trivial class in $H^1(\RP^m)$ under the $i$-th projection map $(\RP^m)^s\to\RP^m$. Note that we do not stress the dependence of $x_i$ on $s$. This is because, if $s'>s$ and $\pi_{s,s'}\colon(\RP^m)^{s'}\to(\RP^m)^{s}$ is the projection onto the first $s$ coordinates, then we think of the map induced in cohomology by $\pi_{s,s'}$ as a honest inclusion. Note that, in these conditions, the standard (graded) basis of $H^*((\RP^m)^s)$ consists of the monomials $x_1^{e_1}x_2^{e_2}\cdots x_s^{e_s}$ where $0\leq e_i\leq m$ ---recall that $x_i^{m+1}=0$.

\begin{proof}[Proof of Theorem~\ref{nataliathm}]
Let $s=r(m)$ and $s_\ell=1+\sum_{i=0}^{\ell-1}r_i$ for $0\leq\ell\leq t$. Consider the product of $s$-th zero-divisors
\begin{equation}\label{formulota}
\prod_{\ell=0}^{t}\left(
\prod_{\,i_\ell=1}^{r_\ell}(x_1+x_{s_\ell+i_\ell})^{m+d_\ell} \right).
\end{equation}
By Proposition~\ref{ulbTCn}, it suffices to check that the expansion of~(\ref{formulota}) in terms of the standard basis of $H^*((\RP^m)^s)$ contains the basis element $x_1^{m-(2^e-1)}x_2^m\cdots x_s^m$.

\smallskip
Note that the $\ell$-th factor in~(\ref{formulota}) is to be neglected if $r_\ell=0$ and, by construction, this happens whenever $\binom{m+d_\ell}{d_\ell}$ is even. On the other hand, if $r_\ell>0$ (so that $\binom{m+d_\ell}{d_\ell}$ is odd), then each of the factors $(x_1+x_{s_\ell+i_\ell})^{m+d_\ell}$ in~(\ref{formulota}) takes the form
$$
(x_1+x_{s_\ell+i_\ell})^{m+d_\ell}=x_1^{d_\ell} x_{s_\ell+i_\ell}^m+\textit{monomials involving powers $x_{s_\ell+i_\ell}^p$ with $p<m$.}
$$
Therefore, for the purpose of keeping track of basis elements of the form $x_1^{a_1}x_2^{a_2}\cdots x_s^{a_s}$ with $a_i=m$ for $2\leq i\leq s$,~(\ref{formulota}) becomes
$$
\prod_{\ell=0}^{t}\left(
\prod_{\,i_\ell=1}^{r_\ell}\left(x_1^{d_\ell} \cdot x_{s_\ell+i_\ell}^m\right) \right)=x_1^{d_0r_0+d_1r_1+\cdots+d_tr_t}x_2^m\cdots x_s^m.
$$
The result then follows from~(\ref{rt}).
\end{proof}

\begin{proof}[Proof of Theorem~\ref{estabilizacion}]
As above, set $e=e(m)$. We have mentioned that the work in~\cite{GGL} gives $G(m)\leq2^{e-1}$. Furthermore~\cite[Lemma~4.2]{GGL} shows that the ideal of $s$-th zero-divisors of $\RP^m$ is generated by the classes $x_1+x_i$ with $2\leq i\leq s$. Thus, it suffices to show that no non-zero product
\begin{equation}\label{ntp}
(x_1+x_2)^{m+i_2}(x_1+x_3)^{m+i_3}\cdots(x_1+x_s)^{m+i_s},\mbox{with $s\geq2$ and $m+i_j\geq0$,}
\end{equation}
can yield a gap $G_s(m)$ smaller than $2^e-1$. In view of~(\ref{colaE})
\begin{equation}\label{m+1}
m+1=2^eq
\end{equation}
for some positive odd integer $q$. Suppose to the contrary that there is a non-zero product~(\ref{ntp}) with $sm-\sum_{j=2}^s(m+i_j)<2^e-1$ or, equivalently, with
\begin{equation}\label{condacot}
m<2^e-1+\sum^s_{j=2}i_j.
\end{equation}
It can be assumed in addition that each $i_j$ is positive, for otherwise we just remove the corresponding factor $(x_1+x_j)^{m+i_j}$ from~(\ref{ntp}) without altering~(\ref{condacot}). In this setting, we have that
\begin{equation}\label{condivi}
\mbox{$x_1^{(u+1)2^e}$ divides $\hspace{.5mm}(x_1+x_j)^{m+i_j}$, if $\hspace{.5mm}i_j>2^e u\hspace{.5mm}$ for some $u\geq0$,} 
\end{equation}
for in fact $(x_1+x_j)^{m+i_j}=(x_1+x_j)^{m+1}(x_1+x_j)^{2^e u}(x_1+x_j)^{i_j-2^e u-1}$, where~(\ref{m+1}) gives
$$
(x_1+x_j)^{m+1}(x_1+x_j)^{2^e u}=(x_1^{2^e}+x_j^{2^e})^{q+u},
$$
which is divisible by $x_1^{(u+1)2^e}$ as $x_j^{2^eq}=x_j^{m+1}=0$.

\smallskip
Now, for $c\geq1$, let $p_c$ be the number of integers $i_2,\ldots,i_s$ in~(\ref{ntp}) that lie in the interval $$\{2^e(c-1)+1,2^e(c-1)+2,\ldots,2^ec\}.$$
Then~(\ref{m+1}) and~(\ref{condacot}) yield
$$
2^e q=m+1<2^e+\sum_{c\geq1}2^ec\,p_c\quad\mbox{i.e.}\quad q\leq\sum_c c \, p_c.
$$
The punch line is that~(\ref{condivi}) implies that~(\ref{ntp}) is divisible by $x_1^{\Sigma}$ where $$\Sigma=\sum_c 2^e c \, p_c\geq2^eq=m+1$$ which, in view of the relation $x_1^{m+1}=0$, contradicts the non-triviality of~(\ref{ntp}).
\end{proof}

\section{Binary expansions}\label{sectionbehavior}
In this section we illustrate the way in which the values of the function $r(m)$ depend on the number and distribution of ones in the binary expansion of $m$. With this in mind, it is convenient to set a suitably flexible notation.

\begin{definition}\label{ksdiwjdnc}
Let $m$ be a positive integer. Write $m=\sum_{i=0}^{\mu}b_i2^{i}\;$ with $\,b_i\in\{0,1\}$ and $b_{\mu}=1$. The binary expansion of $m$, that is, the string of zeros and ones $b_\mu\hspace{.4mm}b_{\mu-1}\cdots b_0$, starts (from the left) with a block of ones, say $n_1$ of them; then it has a block of zeros, say $z_1$ of them; then it has a second block of ones, say $n_2$ of them, and so forth. The \emph{codified binary expansion (cbe)} of $m$ is the (finite) sequence of positive integers $\cbe(m)=(n_1,z_1,n_2,\ldots)$. Note that the length of the sequence $\cbe(m)$ agrees mod 2 with $m$, and that $\mu$ is the integral part of $\log_2(n)$. It is standard to set $\alpha:=\sum n_i$ (the number of ones in the binary expansion of $m$) and $\nu:=\min\{i\colon b_i\neq0\}$ (the exponent in the highest 2-power dividing $m$). For instance, $\nu=z_\omega$ when $\cbe(m)=(n_1,z_1,\ldots,n_\omega,z_\omega)$. If we need to stress the dependence of the parameters $e$, $\alpha$, $\mu$, $\nu$, $b_i$, $n_i$, or $z_i$ on $m$, we use the notation $e(m)$, $\alpha(m)$, $\mu(m)$, $\nu(m)$, $b_i(m)$, $n_i(m)$, or $z_i(m)$, accordingly. The relation $\cbe(m)=(n_1,z_1,n_2,\ldots)$ sets a bijective correspondence from the set of positive integers $m$ to the set of finite sequences of positive integers $(n_1,z_1,n_2,\ldots)$, and we use $p_2(n_1,z_1,n_2,\cdots)=m$ for the inverse function. In fact, it will be convenient to replace the notation $p_2(n_1,z_1,n_2,\cdots)$ by the corresponding binary expansion $1^{n_1}0^{z_1}1^{n_2}\cdots$, where exponents indicate the number of times that a zero or a one is to be repeated.
\end{definition}

\begin{proposition}\label{spaced}
Let $m$ be even with $\cbe(m)=(n_1,z_1,\ldots,n_\omega,z_\omega)$ and $n_1<n_2<\cdots<n_\omega$. Assume $n_u<z_u$ for $1\leq u\leq\omega$ (this condition can be thought of as saying that the blocks of ones in the binary expansion of $m$ are ``suitably'' spaced). Then $r(m)=1+2^{n_\omega}$. More explicitly, the non-zero numbers $r_\ell$ $(0\leq\ell\leq t)$ in Definition~\ref{nataliadef} hold for $\ell\in\{\kappa_u,\ell_u\colon 1\leq u\leq \omega\}$ where
\begin{align*}
\kappa_u&=1^{z_1}0^{n_2}1^{z_2}\cdots0^{n_{u-1}}1^{z_{u-1}}0^{n_u+z_u+\cdots+n_\omega+z_\omega};\\
\ell_u&=1^{z_1}0^{n_2}1^{z_2}\cdots0^{n_{u-1}}1^{z_{u-1}}0^{n_u}1^{z_u-n_u}0^{n_u+n_{u+1}+z_{u+1}+\cdots+n_\omega+z_\omega}.
\end{align*}
(Just as the sum $\sum_{i=u+1}^{\omega}(n_{i}+z_{i})$ is ignored for $u=\omega$, the initial segment $1^{z_1}0^{n_2}1^{z_2}\cdots0^{n_{u-1}}1^{z_{u-1}}$ in the two binary expansions above should be ignored for $u=1$. For instance $\kappa_1=0$.) Furthermore
\begin{align*}
r_{\kappa_1}=2^{n_1}-1 \mbox{ \ with \ } d_{\kappa_1}={}&\,1^{z_1}0^{n_2}1^{z_2}\cdots0^{n_\omega}1^{z_\omega},\\
r_{\ell_1}=1 \mbox{ \ with \ }d_{\ell_1}={}&\,1^{n_1}0^{n_2}1^{z_2}\cdots0^{n_\omega}1^{z_\omega},
\end{align*}
and, for $u\geq2$,
\begin{align*}
r_{\kappa_u}=2^{n_u}-2^{n_{u-1}}-1 \mbox{ \ with \ } d_{\kappa_u}={}&\,1^{z_u}0^{n_{u+1}}1^{z_{u+1}}\cdots0^{n_\omega}1^{z_\omega},\\
r_{\ell_u}=1 \mbox{ \ with \ }d_{\ell_u}={}&\,1^{n_u}0^{n_{u+1}}1^{z_{u+1}}\cdots0^{n_\omega}1^{z_\omega}.
\end{align*}
\end{proposition}
\begin{proof}
The assertion following Definition~\ref{nataliadef} obviously generalizes to the observation that, for any $u\in\{1,\ldots,\omega\}$,  the binary expansions of $d_{\kappa_u}$ and $d_{\ell_u}$ are complementary to that of $m$. In particular all the binomial coefficients $\binom{m+d_{\kappa_u}}{d_{\kappa_u}}$ and $\binom{m+d_{\ell_u}}{d_{\ell_u}}$ with $u\in\{1,\ldots,\omega\}$ are odd.

\smallskip
We start by considering in detail the (slightly special) case $u=1$. The equality $r_0=2^{n_1}-1$ follows from the fact that $(2^{n_1}-1)d_0\leq m<2^{n_1}d_0$, which in turns holds since
\begin{align}
m-(2^{n_1}-1)d_0&=1^{n_1}0^{z_1}1^{n_2}0^{z_2}\cdots1^{n_\omega}0^{z_\omega}-1^{z_1}0^{n_2}1^{z_2}\cdots0^{n_\omega}1^{z_\omega}0^{n_1}+1^{z_1}0^{n_2}1^{z_2}\cdots0^{n_\omega}1^{z_\omega}\nonumber\\
&=1^{n_1+z_1+n_2+z_2+\cdots+n_\omega+z_\omega}-1^{z_1}0^{n_2}1^{z_2}\cdots0^{n_\omega}1^{z_\omega}0^{n_1}\nonumber\\
&=1^{n_2}0^{z_2}\cdots1^{n_\omega}0^{z_\omega}1^{n_1}\geq0\label{mhos}
\end{align}
and $m-2^{n_1}d_0=1^{n_2}0^{z_2}\cdots1^{n_\omega}0^{z_\omega}1^{n_1}-1^{z_1}0^{n_2}1^{z_2}\cdots0^{n_\omega}1^{z_\omega}<0$, due to the assumption $n_1<z_1$.

Next we show that
\begin{equation}\label{loop1}
r_\ell=0 \;\mbox{ for } \;0<\ell<\ell_1.
\end{equation}
For such a value of $\ell$ we have
\begin{equation}\label{recia1}
0^{n_1}1^{z_1}0^{n_2}1^{z_2}\cdots0^{n_\omega}1^{z_\omega}=d_0>d_\ell=d_0-\ell>d_0-\ell_1=0^{z_1}1^{n_1}0^{n_2}1^{z_2}\cdots0^{n_\omega}1^{z_\omega},
\end{equation}
so that the binary expansion of $d_\ell$ must have at least one of the $0$'s on the right of~(\ref{recia1}) changed to a~$1$. If such a 1 appears in one of the blocks $0^{n_i}$ with $2\leq i\leq\omega$, then the binomial coefficient $\binom{m+d_\ell}{d_\ell}$ is obviously even, and so $r_\ell=0$. Otherwise, the 1 must appear in the block $0^{z_1}$, so that $d_\ell\geq2^{n_1+n_2+z_2+\cdots+n_\omega+z_\omega}$. In such a situation the vanishing of $r_\ell$ follows from the easy-to-check fact that $2^{n_1+n_2+z_2+\cdots+n_\omega+z_\omega}>m-d_0r_0$.

For the case $\ell=\ell_1$, note that
\begin{equation}\label{reaesp}
d_{\ell_1}\leq m-d_0r_0<2 d_{\ell_1},
\end{equation}
which yields $r_{\ell_1}=1$. The first inequality in~(\ref{reaesp}) holds since
$$
d_{\ell_1}+2^{n_1}d_0=0^{z_1}1^{n_1}0^{n_2}1^{z_2}\cdots0^{n_\omega}1^{z_\omega}+1^{z_1}0^{n_2}1^{z_2}\cdots0^{n_\omega}1^{z_\omega}0^{n_1}\leq1^{n_1+z_1+\cdots+n_\omega+z_\omega}=m+d_0,
$$
where the inequality comes from the fact that both $0^{z_1}1^{n_1}0^{n_2}\cdots$ and $1^{z_1}0^{n_2}\cdots$ have zeros in their $(n_1+z_1+1)$-st position counted from the left, so that any previous carry in the binary sum disappears at that spot, while no further carries appear from that point on. The second inequality in~(\ref{reaesp}) holds since
$$
m+d_0=1^{n_1+z_1+\cdots+n_\omega+z_\omega}<
0^{z_1-1}1^{n_1}0^{n_2}1^{z_2}\cdots0^{n_\omega}1^{z_\omega}0+1^{z_1}0^{n_2}1^{z_2}\cdots0^{n_\omega}1^{z_\omega}0^{n_1}=2d_{\ell_1}+2^{n_1}d_0,
$$
where the inequality is due to the fact that a carry is forced at the end of the binary sum of $0^{z_1-1}1^{n_1}\cdots$ and $1^{z_1}0^{n_2}\cdots$. 

It is convenient to note at this point that the numerator in the quotient defining the next non-trivial $r_\ell$ ($\ell>\ell_1$) is
\begin{align}
m-(2^{n_1}-1)d_0-d_{\ell_1}&=m-1^{z_1}0^{n_2}1^{z_2}\cdots0^{n_\omega}1^{z_\omega}0^{n_1}+1^{z_1}0^{n_2}1^{z_2}\cdots0^{n_\omega}1^{z_\omega}-1^{n_1}0^{n_2}1^{z_2}\cdots0^{n_\omega}1^{z_\omega}\nonumber\\
&=m-1^{z_1}0^{n_2}1^{z_2}\cdots0^{n_\omega}1^{z_\omega}0^{n_1}+1^{z_1-n_1}0^{n_1+n_2+z_2+\cdots+n_\omega+z_\omega}\nonumber\\
&=1^{n_1}0^{z_1}1^{n_2}0^{z_2}\cdots1^{n_\omega}0^{z_\omega}-1^{n_1}0^{z_1-n_1}0^{n_2}1^{z_2}\cdots0^{n_\omega}1^{z_\omega}0^{n_1}.\label{relevant}
\end{align}

The case $u=1$ will be complete once we show that $r_\ell=0$ for $\ell_1<\ell<k_2$. For such a value of $\ell$ we have
\begin{equation}\label{recia2}
0^{z_1}1^{n_1}0^{n_2}1^{z_2}\cdots0^{n_\omega}1^{z_\omega}=d_0-\ell_1>d_\ell=d_0-\ell>d_0-\kappa_2=0^{n_1+z_1}0^{n_2}1^{z_2}\cdots0^{n_\omega}1^{z_\omega},
\end{equation}
so that the binary expansion of $d_\ell$ must have at least one of the $0$'s on the right of~(\ref{recia2}) changed to a~$1$. As above, if such a 1 appears in one of the blocks $0^{n_i}$ with $2\leq i\leq\omega$, then the binomial coefficient $\binom{m+d_\ell}{d_\ell}$ is obviously even, and so $r_\ell=0$. Otherwise, the 1 must appear in the block $0^{n_1+z_1}$, so that $d_\ell\geq2^{n_2+z_2+\cdots+n_\omega+z_\omega}$. In such a situation the vanishing of $r_\ell$ follows from the fact that $2^{n_2+z_2+\cdots+n_\omega+z_\omega}$ is strictly larger than~(\ref{relevant}), which in turn is observed from the binary-sum setup below.
\begin{align*}
\overbrace{1^{n_1}0^{z_1-n_1}}^{z_1}\overbrace{0^{n_2}1^{z_2}\cdots0^{n_\omega}1^{z_\omega}}^{n_2+z_2+\cdots+n_\omega+z_\omega}\overbrace{0^{n_1}}^{n_1}&\\
\mbox{\large $+$\hspace{.8cm}}\underbrace{\rule{1mm}{0mm}}_{n_1}\underbrace{\hspace{1.2cm}\underbrace{\rule{4mm}{0mm}1}_{n_1}}_{z_1}\underbrace{0^{n_2}0^{z_2}\cdots0^{n_\omega}0^{z_\omega}}_{n_2+z_2+\cdots+n_\omega+z_\omega}\\ \raisebox{1.5mm}{\rule{6.6cm}{.2mm}}\\
1^{n_1}0^{z_1}\cdots \quad < \quad 1^{n_1}\!\underbrace{000\cdots0}_{z_1-1}1\cdots\hspace{2.36cm}
\end{align*}

The cases $u\geq2$ can now be dealt with recursively, using part of the previous analysis. For the start of the recursion we have to check that $r_{\kappa_2}=2^{n_2}-2^{n_1}-1$ or, equivalently, that
$$
(2^{n_2}-2^{n_1}-1)d_{\kappa_2}\leq m-(2^{n_1}-1)d_0-d_{\ell_1}<(2^{n_2}-2^{n_1})d_{\kappa_2}.
$$
Again, both inequalities are verifiable from the corresponding binary expansions. Indeed,~(\ref{relevant}) yields
\begin{align}
m-&(2^{n_1}-1)d_0-d_{\ell_1}-(2^{n_2}-2^{n_1}-1)d_{\kappa_2}\nonumber\\
&=1^{n_1}0^{z_1}1^{n_2}0^{z_2}\cdots1^{n_\omega}0^{z_\omega}-1^{n_1}0^{z_1-n_1}0^{n_2}1^{z_2}\cdots0^{n_\omega}1^{z_\omega}0^{n_1}-(2^{n_2}-2^{n_1}-1)d_{\kappa_2}\nonumber\\
&=1^{n_1}0^{z_1}1^{n_2}0^{z_2}\cdots1^{n_\omega}0^{z_\omega}+0^{n_2}1^{z_2}\cdots0^{n_\omega}1^{z_\omega}-1^{n_1}0^{z_1-n_1}0^{n_2}1^{z_2}\cdots0^{n_\omega}1^{z_\omega}0^{n_1}-(2^{n_2}-2^{n_1})d_{\kappa_2}\nonumber\\
&=1^{n_1}0^{z_1}1^{n_2+z_2+\cdots+n_\omega+z_\omega} -1^{n_1}0^{z_1-n_1}0^{n_2}1^{z_2}\cdots0^{n_\omega}1^{z_\omega}0^{n_1}+0^{n_2}1^{z_2}0^{n_3}1^{z_3}\cdots0^{n_\omega}1^{z_\omega}0^{n_1}-2^{n_2}d_{\kappa_2}\nonumber\\
& = 1^{n_1}0^{z_1}1^{n_2+z_2+\cdots+n_\omega+z_\omega} -1^{n_1}0^{z_1-n_1}0^{n_2+z_2+\cdots+n_\omega+z_\omega+n_1}-2^{n_2}d_{\kappa_2}\nonumber\\
& = 1^{n_1}0^{z_1}1^{n_2+z_2+\cdots+n_\omega+z_\omega} -1^{n_1}0^{z_1+n_2+z_2+\cdots+n_\omega+z_\omega}-1^{z_2}0^{n_3}1^{z_3}\cdots0^{n_\omega}1^{z_\omega}0^{n_2}\nonumber\\
& = 1^{n_1}0^{z_1}1^{n_2+z_2+\cdots+n_\omega+z_\omega} -1^{n_1}0^{z_1}1^{z_2}0^{n_3}1^{z_3}\cdots0^{n_\omega}1^{z_\omega}0^{n_2}\nonumber\\ 
&= 1^{n_3}0^{z_3}\cdots1^{n_\omega}0^{z_\omega}1^{n_2}\geq0\label{repac}
\end{align}
and so
\begin{align*}
m-(2^{n_1}-1)d_0-d_{\ell_1}-(2^{n_2}-2^{n_1})d_{\kappa_2}= 1^{n_3}0^{z_3}\cdots1^{n_\omega}0^{z_\omega}1^{n_2}-1^{z_2}0^{n_3}1^{n_3}\cdots0^{n_\omega}1^{z_\omega} <0.
\end{align*}

Form this point on, the proof enters a recursive loop which starts (for $u=2$) with the fact that the two relations
$$
\mbox{$r_\ell=0$ for $\kappa_u<\ell<\ell_u$ \;\; and \;\; $r_{\ell_u}=2^{n_u}-2^{n_{u-1}}-1$}
$$
are shown for $2\leq u\leq w$ following the arguments proving~(\ref{loop1}) and~(\ref{reaesp}), respectively. In fact, the situation is formally identical as the reader will note by comparing~(\ref{repac}) with~(\ref{mhos}), as well as by comparing the easily verified fact that~(\ref{relevant}) takes the form $1^{n_2}0^{z_2}1^{n_3}0^{z_3}\cdots1^{n_\omega}0^{z_\omega}-0^{n_2-n_1}1^{z_2}0^{n_3}1^{z_3}\cdots0^{n_\omega}1^{z_\omega}0^{n_1}$ with
\begin{align*}
m-(2^{n_1}-1)d_0-d_{\ell_1}-(2^{n_2}-&2^{n_1}-1)d_{\kappa_2}-d_{\ell_2}\\&=1^{n_3}0^{z_3}\cdots1^{n_\omega}0^{z_\omega}1^{n_2}-d_{\ell_2} \qquad\qquad\qquad\qquad\quad\mbox{(by~(\ref{repac}))}\\
&=1^{n_3}0^{z_3}\cdots1^{n_\omega}0^{z_\omega}1^{n_2}-1^{n_2}0^{n_3}1^{z_3}\cdots0^{n_\omega}1^{z_\omega}\\
&=1^{n_3-n_2}0^{z_3}\cdots1^{n_\omega}0^{z_\omega}1^{n_2}-0^{n_3}1^{z_3}\cdots0^{n_\omega}1^{z_\omega}\\
&=1^{n_3}0^{z_3}\cdots1^{n_\omega}0^{z_\omega}-0^{n_3-n_2}1^{z_3}\cdots0^{n_\omega}1^{z_\omega}0^{n_2},
\end{align*}
where the last equality is obtained by complementing with respect to $1^{n_3+z_3+\cdots+n_\omega+z_\omega}$.
\end{proof}

\begin{remark}\label{lightext}{\em
Proposition~\ref{spaced} can be generalized in at least two directions. Firstly, the hypothesis $n_u<z_u$ for $1\leq u\leq\omega$ can be weakened to requiring only $n_u\leq z_u$ for $1\leq u\leq\omega$ without altering the main conclusion $r(m)=1+2^{n_\omega}$. In fact, the explicit description of the non-trivial $r_\ell$'s changes only slightly: Note that $\kappa_u=\ell_u$ (and of course $d_{\kappa_u}=d_{\ell_u}$) whenever $n_u=z_u$, in which case the two values of $r_{\kappa_u}$ and $r_{\ell_u}$ will merge into the single
\begin{equation}\label{fab}
r_{\kappa_u}=r_{\ell_u}=2^{n_u}-2^{n_{u-1}}
\end{equation}
(interpreting $2^{n_0}$ as zero). We leave as an exercise for the reader adapting the proof of Proposition~\ref{spaced} to prove~(\ref{fab}).
}\end{remark}

\begin{remark}\label{spacedgeneralized}{\em
More generally, for $m$ as in Proposition~\ref{spaced} (and with the hypothesis $n_u<z_u$ replaced by the more general requirement $n_u\leq z_u$), assume that the binomial coefficient
\begin{equation}\label{obstrgenel}
\binom{\,m+j\;}{m+i}
\end{equation}
is odd for some integer $i\in\{0,1,\ldots,2^{z_\omega}-1\}$ and $j=(i+1)(2^{n_{\omega}}+1)-2$. Then the proof of Proposition~\ref{spaced} can be adapted to give $r(m+i)=1+2^{n_\omega}$. Note that the hypothesis $i<2^{z_\omega}$ implies that the first $2\omega-2$ terms in $\cbe(m+i)$ agree with those of $\cbe(m)$, though the number and form of the subsequent terms in the two codified binary sequences might bear no relationship to each other. So, this example allows us to identify instances of $m'$ where the suitable spacing conditions in Proposition~\ref{spaced} hold only partially, and yet $r(m')$ has a simple description (see Proposition~\ref{no crecimiento w=2 B}).
}\end{remark}

\begin{remark}\label{fibnnnnaci}{\em
The binomial coefficient~(\ref{obstrgenel}) is odd in many instances. A simple way to see this is by observing that the mod~2 value of~(\ref{obstrgenel}) agrees with that of the binomial coefficient~$\binom{j}{i}$ whenever the condition $i<2^{z_\omega}$ is strengthened to $j<2^{z_\omega}$. (The latter hypothesis can be thought of as requiring that the dyadic ``tail'' $i$ in $m+i$ is ``far enough'' from the last block of ones in the binary expansion of $m$). It is then worth noticing that the mod 2 values of~$\binom{j}{i}$ (as $i$ varies) have an interesting arithmetical behavior. Consider, for simplicity, the case $\omega=1=n_1$, where a nice Fibonacci-type fractal pattern arises for the parity properties of the resulting binomial coefficient $\binom{j}{i}=\binom{3i+1}{i}$. Indeed, if we list the mod 2 values of $\binom{3i+1}{i}$ in the range $2^\ell\leq i\leq2^{\ell+1}-2$ with $i$ even and $\ell\geq 2 $, we get the first $2^{\ell-2}$ numbers in the series~(\ref{seriespaperv}) below, followed by $2^{\ell-2}$ zeros.
\begin{equation}\label{seriespaperv}
1^3, 0, 1^2, 0^2, 1^3, 0^5, 1^3, 0, 1^2, 0^{10}, 1^3, 0, 1^2, 0^2, 1^3, 0^{21}, 
1^3, 0, 1^2, 0^2, 1^3, 0^5, 1^3, 0, 1^2, 0^{42}, \ldots
\end{equation}
Here the notation ``$a^b\hspace{.2mm}$'' stands for ``$a,a,\ldots,a$'' where $a$ is repeated $b$ times. The Fibonacci-type behavior enters as follows: Let $f_c$ denote the sequence of the first $2^c$ digits in~(\ref{seriespaperv}). For instance $f_0=(1)$ and $f_1=(1,1)$. Then, for $c\geq2$, $f_c$ is the concatenation of $f_{c-1}$ followed by $f_{c-2}$, and followed finally by $2^{c-2}$ zeros.
}\end{remark}

\begin{remark}\label{motivaciondelotropaper}{\em
The case $i=0$ in Remark~\ref{spacedgeneralized} specializes to Proposition~\ref{spaced}. In turn, the case $n_1=\omega=1$ in Proposition~\ref{spaced} was obtained in~\cite[Theorem~4.3]{GGGL} as part of a series of sharp results for the higher topological complexity of certain families of flag manifolds. In fact, it was the stabilization phenomenon ``$\TC_s\to s\hdim$'' noted in~\cite{GGGL} for flag manifolds that motivated the authors of this paper to take a closer look at the situation for real projective spaces.
}\end{remark}

Other instances in the explicit description of the function $r(m)$ are described in the remainder of this section (for $m$ even), with attention restricted to cases where the binary expansion of $m$ has at most two blocks of ones. For instance, Proposition~\ref{spaced} (as generalized in Remark~\ref{lightext}) and Proposition~\ref{pocadivisibilidad} below account for a full description of the function $r(m)$ when $m$ is even and has a single block of ones. The proofs of Propositions~\ref{pocadivisibilidad} and~\ref{pocadivisibilidad w=2}--\ref{noespaciamientow=2} below follow the same strategy as that used in the proof of Proposition~\ref{spaced}; however the actual arguments are much easier and, consequently, will be left as an exercise for the diligent reader. We will only focus on cases where $m$ is even, in particular $e(m)=0$ and $d_{\ell}=d_0-\ell$. Accordingly, we will specify (the binary expansion of)~$d_\ell$, but will omit explicit reference to $\ell$. (Both $\ell$ and $d_\ell$ were described explicitly in the statement of Proposition~\ref{spaced} for proof-referencing purposes.)

\begin{proposition}\label{pocadivisibilidad}
Let $m$ be even with $\cbe(m)=(n,z)$ and $n>z$. Then $r(m)=1+\sum_{i=0}^{\sigma}2^{n-i z}$, where $\sigma$ stands for the largest integer which is strictly smaller than $n/z$. Explicitly, the non-zero numbers $r_\ell$ $(0\leq\ell\leq t)$ in Definition~\ref{nataliadef} hold for $\ell\in\{\ell_1,\ell_2\}$ with $r_{\ell_1}=\sum_{i=0}^{\sigma}2^{n-iz}-1$ and $r_{\ell_2}=1$, where $d_{\ell_1}=1^z$ and $d_{\ell_2}=1^{n-\sigma z}$.
\end{proposition}

Just as in Remark~\ref{lightext}, when $n-\sigma z=z$, so that $\ell_1=\ell_2$, the two values $r_{\ell_1}$ and $r_{\ell_2}$ should be interpreted as merging into the single $r_{\ell_1}=r_{\ell_2}=\sum_{i=0}^{\sigma}2^{n-iz}$. A similar phenomenon applies with Propositions~\ref{pocadivisibilidad w=2}--\ref{noespaciamientow=2}, but we will make no further comments on such a direction.

\begin{remark}{\em
The weakest instance in Proposition~\ref{pocadivisibilidad} holds with $z_1=1$ (that is, when $m+2$ is a 2-power), for then $r(m)$ agrees with the linear function $\ell(m)$ in the introduction. This situation should be compared with the observation in Example~\ref{el de la evidencia} regarding the hardness of the immersion problem for $\RP^m$ when $m+1$ is a 2-power.
}\end{remark}

When $m$ is even and $\cbe(m)=(n_1,z_1,n_2,z_2)$, the hypotheses in Proposition~\ref{spaced} (as generalized in Remark~\ref{lightext}) become
\begin{equation}\label{three4w=2}
\mbox{$n_1<n_2$, $\;n_1\leq z_1$, \;and $\;n_2\leq z_2$.}
\end{equation}
The following results describe the value of $r(m)$ when all but one of these inequalities hold. Note that Proposition~\ref{pocadivisibilidad w=2} below is analogous to Proposition~\ref{pocadivisibilidad}, while Proposition~\ref{no crecimiento w=2 B} below fits in the setting of Remark~\ref{spacedgeneralized} (as simplified in Remark~\ref{fibnnnnaci}).

\begin{proposition}\label{pocadivisibilidad w=2}
Let $m$ be even with $\cbe(m)=(n_1,z_1,n_2,z_2)$, $n_1\leq z_1$ and $\max\{n_1,z_2\}<n_2$. Then $r(m)=1+\sum_{i=0}^{\sigma}2^{n_2-i z_2}$ where $\sigma$ stands for the largest integer which is strictly smaller than $n_2/z_2$. Explicitly, the non-zero numbers $r_\ell$ $(0\leq\ell\leq t)$ in Definition~\ref{nataliadef} hold for $\ell\in\{\kappa_1,\ell_1,\kappa_2,\ell_2\}$ with $r_{\kappa_1}=2^{n_1}-1$, $r_{\kappa_2}=\sum_{i=0}^{\sigma}2^{n_2-iz_2}-2^{n_1}-1$, and $r_{\ell_1}=r_{\ell_2}=1$, where $d_{\kappa_1}=1^{z_1}0^{n_2}1^{z_2}$, $d_{\kappa_2}=1^{z_2}$, $d_{\ell_1}=1^{n_1}0^{n_2}1^{z_2}$, and $d_{\ell_2}=1^{n_2-\sigma z_2}$.
\end{proposition}

\begin{proposition}\label{no crecimiento w=2 B}
Let $m$ be even with $\cbe(m)=(n_1,z_1,n_2,z_2)$ and $n_2\leq n_1\leq \min\{z_1,z_2\}$. Then $r(m)=1+2^{n_1}$. Explicitly, the non-zero numbers $r_\ell$ $(0\leq\ell\leq t)$ in Definition~\ref{nataliadef} hold for $\ell\in\{\kappa_1,\ell_1\}$ with $r_{\kappa_1}=2^{n_1}-1$ and $r_{\ell_1}=1$, where $d_{\kappa_1}=1^{z_1}0^{n_2}1^{z_2}$ and $d_{\ell_1}=1^{n_2}0^{z_2}1^{n_1}$.
\end{proposition}

\begin{proposition}
Let $m$ be even with $\cbe(m)=(n_1,z_1,n_2,z_2)$ and $n_2\leq z_2<n_1\leq z_1$. Then $r(m)=1+2^{n_1}+2^{\min\{n_2,n_1-z_2\}}$. Explicitly, the non-zero numbers $r_\ell$ $(0\leq\ell\leq t)$ in Definition~\ref{nataliadef} hold for $\ell\in\{\kappa_1,\ell_1,\kappa_2,\ell_2\}$ with $r_{\kappa_1}=2^{n_1}-1$, $r_{\kappa_2}=2^{\min\{n_2,n_1-z_2\}}-1$, $r_{\ell_1}=r_{\ell_2}=1$, where $d_{\kappa_1}=1^{z_1}0^{n_2}1^{z_2}$, $d_{\kappa_2}=1^{z_2}$, and 
$$\begin{matrix*}[l]
d_{\ell_1}=1^{n_2}0^{n_1}1^{z_2}\quad&\mbox{and}\quad d_{\ell_2}=1^{n_1-z_2},\quad&\mbox{provided}\quad n_2\geq n_1-z_2;\\
d_{\ell_1}=1^{n_2}0^{z_2}1^{n_1-n_2-z_2}0^{n_2}1^{z_2}\quad&\mbox{and}\quad d_{\ell_2}=1^{n_2},\quad&\mbox{provided}\quad n_2\leq n_1-z_2.
\end{matrix*}$$
\end{proposition}

\begin{proposition}\label{noespaciamientow=2}
Let $m$ be even with $cbe(m)=(n_1,z_1,n_2,z_2)$ and $z_1<n_1<n_2\leq z_2$. Then $r(m)=1+2^{n_2}$. Explicitly, the non-zero numbers $r_\ell$ $(0\leq\ell\leq t)$ in Definition~\ref{nataliadef} hold for $\ell\in\{\kappa_1,\ell_1,\kappa_2,\ell_2\}$ with $r_{\kappa_1}=2^{n_2}-r_{\kappa_2}-2$, $$r_{\kappa_2}=2^{n_1+1}(2^{n_2-n_1-1}-1)+2^{\rho}(2^{z_1-1}-1)\sum_{i=0}^q2^{i z_1+2}+2^{\rho+1}-1,$$ and $r_{\ell_1}=r_{\ell_2}=1$, where $d_{\kappa_1}=1^{z_1}0^{n_2}1^{z_2}$, $d_{\kappa_2}=1^{z_2}$, $d_{\ell_1}=1^{\rho+1}0^{n_2}1^{z_2}$, and $d_{\ell_2}=1^{n_2}$. Here $q$ and $\rho$ stand, respectively, for the quotient and remainder in the division of $n_1-z_1-1$ by $z_1$.
\end{proposition}

\section{Immersion dimension via higher TC: an example}
\label{sectionapproach}
In this section we illustrate the $\TC_s$-program devised in the introduction of this paper to approach $\Imm(\RP^m)$. Let us start by considering small dimensional examples, and the way they become part of larger families sharing similar properties.

\medskip
There are three singular situations: $\RP^1$ is a circle, and it certainly fits in the well known description of the higher topological complexity of spheres, where the dimension of the sphere plays the decisive role: $\TC_s(S^{2k})=s$ while $\TC_s(S^{2k+1})=s-1$. Closely related is the case of the $H$-spaces $\RP^3$ and $\RP^7$, whose higher topological complexity has been described in Example~\ref{hopfspaces}.

\medskip
The first truly interesting case is that of the projective plane, which immerses optimally in three-dimensional Euclidean space as the Boy Surface, so $\TC_2(\RP^2)=3$. Note that this is just one below the dimensional bound in Proposition~\ref{ulbTCn}, which contrasts with the fact (from Theorem~\ref{nataliathm}) that
\begin{equation}\label{planoproyectivo}
\mbox{$\TC_s(\RP^2)=2s\,$ for any $s\geq3$.}
\end{equation}
It is worth remarking that~(\ref{planoproyectivo}) is part of a more general phenomenon: Any closed (orientable or not) surface $S$, other than the sphere and the torus, has $\TC_s(S)=2s$ whenever $s\geq3$ (c.f.~\cite[Theorem~5.1]{GGGL}) ---this should also be compared to the fact that the computation of the $\TC_2$-value of the Klein bottle has become such an elusive challenge! For our purposes, a much more interesting observation to make at this point  is that~(\ref{planoproyectivo}) generalizes (again in view of Theorem~\ref{nataliathm}) to the fact that
$$\mbox{
$\TC_s(\RP^{2^a})=2^{a}s$ for any $s\geq3$, while $\TC_2(\RP^{2^a})=\Imm(\RP^{2^a})=2^{a+1}-1$.
}$$
In terms of the $\delta_s$ functions, such a situation translates into the equalities
\begin{equation}\label{sucri1}
\delta_3(2^a)=0 \;\;\mbox{and} \;\; \delta_2(2^a)=1.
\end{equation}
Since $r(2^a)=3$ for $a\geq1$ (Proposition~\ref{spaced}), this yields a nicely regular increasing behavior for the critical sequence~(\ref{sqnmchica}) when $m=2^a$. Admittedly, the length of the sequence~(\ref{sucri1}) is ridiculously short but, as discussed next, a similar regularity phenomenon could actually be holding in the next obvious example, namely $m=2^a+2^{a+1}$ with $a\geq1$ (the special case $m=3$ has been considered above), which we discuss next.

\medskip
At first sight, the situation is slightly special for $m=2^a+2^{a+1}$ if $a=1$, so we consider it first. The immersion dimension of $\RP^6$ is known to be $\TC_2(\RP^6)=7$, while Proposition~\ref{pocadivisibilidad} gives $r(6)=7$. Thus, the critical sequence~(\ref{sqnmchica}) now becomes
\begin{equation}\label{critical6weak}
\delta_7(6)=0,\;\;\delta_6(6)={?},\;\;\delta_5(6)={?},\;\;\delta_4(6)={?},\;\;\delta_3(6)={?},\;\;\delta_2(6)=5.
\end{equation}
Furthermore, the proof of Theorem~\ref{nataliathm} yields $(x_1+x_2)^7(x_2+x_3)^7\cdots(x_1+x_7)^7\neq0$. In particular, if we only consider the first $j-1$ factors ($2\leq j\leq7$), we obtain the last instance in the chain of inequalities $\delta_j(6)\leq G_j(6)\leq7-j$. Consequently,~(\ref{critical6weak}) becomes
\begin{equation}\label{critical6strong}
\delta_7(6)=0,\;\;\delta_6(6)\leq1,\;\;\delta_5(6)\leq2,\;\;\delta_4(6)\leq3,\;\;\delta_3(6)\leq4,\;\;\delta_2(6)=5.
\end{equation}
It would be interesting to know whether~(\ref{critical6strong}) really has the nice steady increasing behavior suggested by~(\ref{sucri1}), namely if
\begin{equation}\label{pretttty}
\delta_j(6)=7-j  \mbox{ \ for \ }2\leq j\leq7.
\end{equation}
For instance,~(\ref{pretttty}) would hold provided one could prove that the inequalities in~(\ref{critical6strong}) held in the stronger form $\delta_i(6)\leq\delta_{i+1}(6)+1$. At any rate, the following considerations are intended to give numerical evidence toward the possibility that~(\ref{pretttty}) holds, in a suitably generalized way, for any $a\geq1$.

\medskip
For $a\geq2$, Proposition~\ref{spaced} gives $r(2^a+2^{a+1})=5$, so the critical sequence~(\ref{sqnmchica}) now takes the slightly shorter form
$$
\delta_5(2^a+2^{a+1})=0,\;\;\delta_4(2^a+2^{a+1})={?},\;\;\delta_3(2^a+2^{a+1})={?},\;\;\delta_2(2^a+2^{a+1})={?}
$$
The currently known information about $\Imm(\RP^{2^a+2^{q+1}})$ for $a\leq5$ yields\footnote{See Davis' tables in~\cite{tables}, which is our main reference for the immersion facts asserted here.}
\begin{description}
\item[\hspace{1cm}Case $a=2$:] \hspace{.3cm} $\delta_2(12)=6$.
\item[\hspace{1cm}Case $a=3$:] \hspace{.3cm} $\delta_2(24)\in\{9,10\}$.\footnote{Note that $\RP^{24}$ is the smallest dimensional projective space whose immersion dimension is not fully known; yet our purely homological methods suffice to get the exact value of $\TC_s(\RP^{24})$ for $s\geq 5$.}
\item[\hspace{1cm}Case $a=4$:] \hspace{.3cm} $\delta_2(48)\in\{9,10,11\}$.
\item[\hspace{1cm}Case $a=5$:] \hspace{.3cm} $\delta_2(96)\in\{13,14,\ldots,18\}$.
\end{description}
The punch line is that the above facts provide some evidence for potentially extending the estimates in~(\ref{critical6strong}) by the following:

\begin{conjecture}\label{monobeha}
For $a\geq1$ and $\hspace{.6mm}2\leq j\leq r(2^a+2^{a+1})$,
\begin{equation}\label{genesinKW}
\delta_j(2^a+2^{a+1})\leq(r(2^a+2^{a+1})-j)a.
\end{equation}
\end{conjecture}

\begin{remark}\label{wleidhtsu}{\em Kitchloo-Wilson's non-immersion result $\RP^{48}\not\subseteq\,\mathbb{R}^{84}$ (the lowest dimensional new  result in~\cite{MR2475624}, which gives $\delta_2(48)<12$) implies that we should not expect equality to hold in~(\ref{genesinKW}) ---which seems to be compatible with the fact that the potential new non-immersion result in Conjecture~\ref{noimmpot} below is still far from the expected optimal~(\ref{belief80s}). 
}\end{remark}

Our interest in the above discussion comes from the fact that Conjecture~\ref{monobeha} obviously contains (with $j=2$) what would be the new (as far as we are aware of) non-immersion result $\delta_2(2^a+2^{a+1})\leq3a$ for $a\geq2$, i.e.:

\begin{conjecture}\label{noimmpot}
For $a\geq2$, $\Imm(\RP^{2^a+2^{a+1}})\geq2^{a+1}+2^{a+2}-3a$.
\end{conjecture}

The truthfulness of the above statement would be particularly interesting in a number of directions. To start, and as noted above, if $a=2$, the non-immersion result in Conjecture~\ref{noimmpot} not only holds true, but it is in fact optimal. A similar situation would hold for $a=3$ as, then, Conjecture~\ref{noimmpot} would actually settle the value of the currently-open immersion dimension of $\RP^{24}$ (one of the iconic cases back in the decade of the 1970's) to be $\TC_2(\RP^{24})=39$. Indeed, Randall's immersion result $$\mbox{$\Imm(\RP^m)\leq2m-9$ \ for $m\equiv0$ mod 8 with $\alpha(m)>1$}$$ would be optimal for $n=24$ due to the ($a=3$)-case in Conjecture~\ref{noimmpot}. For other values of $a$, with the single exception of $a=4$ (see Remark~\ref{wleidhtsu}), Conjecture~\ref{noimmpot} would improve the currently known information on the immersion dimension of $\RP^{2^a+2^{a+1}}$. For instance, the case $a=5$ in Conjecture~\ref{noimmpot} would improve Kitchloo-Wilson's non-immersion result $\Imm(\RP^{96})\geq174$ by three units, and would be within two of being optimal in view of Davis' immersion result $\Imm(\RP^{96})\leq179$.

\medskip
Of course, one could try to apply the $\TC_s$ approach to $\Imm(\RP^m)$ for other families of projective spaces $\RP^m$. For instance, some of the phenomena described above seem to hold for spaces of the form $\RP^{2^a+2^{a+1}+2^{a+2}}\hspace{.4mm}$with $a\geq2$ and, more generally, for spaces $\RP^m$ with $\cbe(m)=(n_1,z_1)$ and $z_1\geq n_1-1$. One could even try to use the same strategy in order to prove (positive) immersion results. Indeed, just as~(\ref{genesinKW}) is a statement about the possibility that the increasing behavior of the critical sequence~(\ref{sqnmchica}) is bounded from above by some linear function, it is natural to try to prove a general statement asserting that, for some fixed integer $\phi(m)$, $\delta_j(m)\geq\delta_{j+1}(m)+\phi(m)$ in the range of the critical sequence~(\ref{sqnmchica}). Such a possibility will most likely need to use stronger homotopy methods (e.g.~the Hopf-type obstruction methods recently developed in~\cite{Hopfpaper}), rather than the homological methods in this paper. For instance, the homotopy obstruction methods in~\cite[Section~2]{MR2649230} seem to lead to a proof of equality in~(\ref{genesinKW}) for $j=r(2^a+2^{a+1})$. Details will appear elsewhere.

\bibliographystyle{plain}
\bibliography{stab}

\bigskip
{\sc Departamento de Matem\'aticas

Centro de Investigaci\'on y de Estudios Avanzados del IPN

Av.~IPN 2508, Zacatenco, M\'exico City 07000, M\'exico

{\tt cadavid@math.cinvestav.mx}

{\tt jesus@math.cinvestav.mx}

{\tt aldo@math.cinvestav.mx}
}
\end{document}